\documentclass[a4paper, 12pt]{article}
\usepackage{latexsym}
\usepackage{comment}
\usepackage{amsfonts, amssymb,  mathtools, enumerate, xcolor}
\usepackage{geometry}
\usepackage{amsmath, amsfonts}
\usepackage{amsthm}
\usepackage{placeins}
\usepackage{setspace}
\usepackage{graphicx}
\usepackage{tabularx}
\usepackage{mathrsfs}
\numberwithin{equation}{section}
\usepackage{cite}
\usepackage{epstopdf}
\usepackage{csquotes}
\usepackage{tikz}
\usepackage{caption}
\usepackage{subfig}
\usepackage{parskip}
\usepackage[colorlinks=true, linkcolor=red, citecolor=blue, filecolor=blue, urlcolor=blue]{hyperref}
\newtheorem{theorem}{Theorem}[section]
\newtheorem{remark}{Remark}[section]
\newtheorem{proposition}{Proposition}[section]
\newtheorem{lemma}{Lemma}[section]

\newtheorem{corollary}{Corollary}[section]
\title{ Asymptotic Behavior and Error Bounds for the Initial-Boundary Value Problems for the Autonomous Fisher-KPP Equation on the Real Half-Line}
\author{Chu Chu\footnote{corresponding author} and M. W. Wong}
\date{Department of Mathematics and Statistics\\ York University\\4700 Keele Street\\ Toronto, Ontario M3J 1P3\\ Canada\\
chuchu7@my.yorku.ca, mwwong@yorku.ca}
\DeclareMathOperator{\sech}{sech}
\begin{document}

\maketitle

\textbf{Abstract}:
We study the initial-boundary value problem for the Fisher-KPP equation on the half-line under the Dirichlet condition, the Neumann condition and the 
Robin condition. For the autonomous logistic equation, we determine explicitly the nonnegative stationary profiles converging to $1$ at infinity and obtain exponential boundary-layer asymptotics. We prove that the nontrivial Neumann solutions converge to $1$, and we obtain a local uniform convergence of the difference of the Neumann solution and the Robin solution, and a quantitative exponential comparison for the initial data close to the corresponding stationary profiles. We also describe the dependence of the positive stationary profile of a general homogeneous Robin data. Finally, for small time-periodic Neumann and Robin boundary forcing, we construct a periodic lifted mild solution that is unique in a fixed neighborhood of the unforced steady state, derive a first-order expansion with a uniform $O(\varepsilon^2)$ remainder in
$C_0[0,\infty)$, which is the space of all continuous functions on $[0,\infty)$ going to 0 as $x\to \infty$  and prove nonlinear exponential attraction of the periodic response.



\section{Introduction} 
 % \begin{definition}[semilinear]
  %it is linear in the highest-order derivatives.
%  \end{definition}

The initial-boundary value problem for the semilinear heat equation with the Dirichlet condition on $[0,\infty)$ is of the form 
$$\left\{\begin{array}{lll}\frac{\partial u}{\partial t}(x,t)-\frac{\partial^2u}{\partial x^2} (x,t)=f(u(x,t)), \quad x>0,\quad t>0,\\u(0,t)=0, \quad t\geq 0,\\
u(x,0)=\phi(x),\quad 0\leq x < \infty.\end{array}\right.$$

The initial-boundary problem for the semilinear heat equation with the Neumann condition on $[0, \infty)$ is of the form
$$\left\{\begin{array}{lll}
\frac{\partial u}{\partial t}(x,t)-\frac{\partial^2u}{\partial x^2}(x,t)=f(u(x,t)),\quad x>0,\quad t>0,\\
u_+'(0,t)=0,\quad t>0,\\
u(x,0)=\phi(x),\quad 0\leq x<\infty,\end{array}\right.$$
where $$u_+'(0,t)=\lim_{x\rightarrow 0^+}\frac{u(x,t)-u(0,t)}{x}$$
for all $t$ in $(0,\infty).$
Finally, the initial-boundary value problem for the semilinear heat  equation with the Robin boundary condition on  $[0, \infty)$ is of the form
$$\left\{\begin{array}{lll}
\frac{\partial u}{\partial t}(x,t)-\frac{\partial^2u}{\partial x^2}(x,t)=f(u(x,t)),\quad x>0,\quad t>0,\\
u_+'(0,t)= u(0,t),\quad t>0,\\
u(x,0)=\phi(x), \quad 0\leq x< \infty.\end{array}\right.$$

In the above three initial-boundary value problems, the initial value $\phi$ is taken to  be a well-behaved function on $[0,\infty)$, which is to be made precise as the need arises.  
The Fisher--KPP equation goes back to the classical works of Fisher \cite{AK} and Kolmogorov--Petrovskii--Piskunov \cite{BK}. See also Aronson--Weinberger \cite{FK} for classical results on spreading phenomena. In this paper we use the logistic Fisher--KPP nonlinearity
$$
f(u)=u(1-u),\qquad 0\le u\le 1.
$$
More generally, the standard Fisher--KPP assumptions are
$$
f(0)=f(1)=0,$$ $$f'(1)\le 0\le f'(0),$$
$$
f'(u)\le f'(0),\quad u\in(0,1),
$$
and
$$
f(u)>0,\quad u\in(0,1).
$$

The time-periodic Dirichlet problem studied in \cite{CXC} is the genuinely
nonautonomous equation
$$
u_t=u_{xx}+f(t,u),\quad x\geq 0,\,t\geq 0,$$
with
$$
f(t+T,u)=f(t,u), \quad t>0,\, 0\leq u\leq 1,
$$
and is distinct from the autonomous logistic equation considered
in this paper. Accordingly, we do not use \cite{CXC} to infer a non-constant
time-periodic solution of the autonomous homogeneous Dirichlet
problem. In the autonomous analysis, the Dirichlet condition enters
only through the stationary profile $V_D$. For the autonomous homogeneous Dirichlet problem, the result of Yi and
Zhao \cite[Section~6.3, especially Corollary~6.3]{YZ} is used here only for the existence and uniqueness of the nontrivial Dirichlet stationary profile and for convergence toward that profile, as specified below. To match the nonnegative-range hypothesis imposed there on the birth function, define the function $\widetilde{f}$ by
$$
\widetilde f(u)
=
\begin{cases}
2u-u^2, & 0\leq u\leq 1,\\
1,      & u\geq 1.
\end{cases}
$$
Then
$$
\widetilde f:[0,\infty) \to [0,\infty)$$
and
$$
\widetilde f'(0)=2.
$$
Moreover, $u^*=1$ is its unique positive fixed point. Setting
$\tau=0$, $d=\mu=1$, and replacing the birth function in
(6.13) of Yi and Zhao \cite{YZ} by $\widetilde f$, we obtain
$$
u_t=u_{xx}-u+\widetilde f(u),
\quad x>0,\, t>0,
$$
with the homogeneous Dirichlet boundary condition $$u(0,t)=0,\quad t\geq 0.$$For initial data satisfying
$$
0\leq \phi(x)\leq 1,\quad x\geq 0,$$
$$
\phi(0)=0$$
and
$$\phi(x)\neq  0$$ for some $x\in [0,\infty),$
the comparison principle keeps the solution in the interval $[0,1]$.
On this interval,
$$
-u+\widetilde f(u)=u(1-u),
$$
so the equation coincides with the autonomous logistic Fisher--KPP
problem considered in  this paper. Consequently, Corollary~6.3 of Yi and Zhao
establishes the existence and uniqueness of the nontrivial stationary
profile connecting $0$ to $1$, together with convergence to this
profile uniformly for
$$
0\leq x\leq(2-\varepsilon)t
$$
with every $\varepsilon\in(0,2)$. The overlap with Yi and Zhao is therefore limited to these two Dirichlet conclusions. Section~6.3 of their work does not treat the Neumann or Robin half-line problems considered in this paper, nor does it contain the explicit stationary-profile formulas, spatial asymptotics, quantitative Neumann--Robin comparisons, or the periodic-response and nonlinear-attraction  established in this paper. Whenever  a Robin solution converges uniformly on every compact subset of $[0,\infty)$ to a stationary profile is needed, it is stated explicitly under the hypotheses.

The comparison principle and the finite-interval logistic stationary problem used in Lemma~\ref{lem:neumann-convergence} are standard. See, for example,\cite{AB,EK}. The $C_0$-semigroup and fixed-point techniques used in
Theorem~\ref{thm:periodic-existence} and
Proposition~\ref{prop:periodic-expansion} follow the standard
approach for semilinear parabolic equations. See \cite{CK,IW, Pa,DK}.

We end this Introduction with a gist of the contributions of this paper. For the autonomous
logistic equation, we determine the Dirichlet and Robin boundary
profiles explicitly, obtain their exponential decay asymptotics, and
describe the monotone dependence of the positive profile on a slightly more general
homogeneous Robin problem. In addition to the local
comparison valid for solutions already known to converge, we establish
a quantitative exponential Neumann-Robin comparison for initial data
in a neighborhood of the corresponding stationary profiles. For
small time-periodic Neumann and Robin boundary forcing, a boundary
lifting and an exponentially stable homogeneous semigroup yield a
periodic lifted mild solution that is unique in a fixed neighborhood of the steady state. We obtain a first-order response expansion with a uniform quadratic remainder and prove nonlinear exponential attraction of this periodic state. The Dirichlet periodic-forcing problem is not included, because its nonhomogeneous boundary input requires a different phase-space or maximal-regularity formulation.

Lucid expositions of the Fisher--KKP equation can be found in \cite{LL}. See also \cite{STR} for discussions on the (linear) heat equation with each of the boundary conditions investigated in this paper. 

\section{Asymptotic Comparison and Neumann Convergence}
\label{sec:asymptotic-comparison}
We first clarify why no non-constant time-periodic Dirichlet profile is
used for the autonomous logistic equation. It is due to the gradient structure of the autonomous Dirichlet problem, which is detailed in the following remark.

\begin{remark}
\label{rem:dirichlet-gradient}
{\rm Let $F:[0,\infty) \to \mathbb{R}$ be the function defined by 
$$
F(s)=\frac12s^2-\frac13s^3,\quad s\in (0,\infty). 
$$
Then
$$F'(s)=s(1-s),\quad s\in 0,\infty).$$ Let $u$ be a sufficiently smooth
solution of
$$
\left\{\begin{array}{ll}
u_t=u_{xx}+u(1-u),\\
u(0,t)=0,\end{array}\right. \begin{array}{ll} x\in(0,\infty),\,t>0,\\
                                                 t>0,\end{array}$$
and for a fixed time $t \in (0,\infty)$, assume that $u(x,t)\to a$ in  $\{0,1\}$ as $x\to\infty$ with
sufficient decay for the following energy and the integration by parts
to be well defined. Now, we define ${\mathcal{E}}_a[u(t)]$ for $t\geq 0$ by
$$
\mathcal E_a[u(t)]
 =\int_0^\infty
 \left(
   \frac12|u_x(x,t)|^2-F(u(x,t))+F(a)
 \right)\,dx.
$$
Then
$$
\frac{d}{dt}\mathcal E_a[u(t)]
 =-\int_0^\infty |u_t(x,t)|^2\,dx\le0.
$$
Indeed, the boundary term at $x=0$ vanishes because
$u_t(0,t)=0$. Consequently, if $u$ is $T$-periodic in time,
then
$$
0=\mathcal E_a[u(T)]-\mathcal E_a[u(0)]
 =-\int_0^T\int_0^\infty |u_t(x,t)|^2\,dx\,dt
$$
and hence $$u_t(x,t)=0,\quad x\in[0,\infty),\,t>0.$$ Thus, every time-periodic solution is stationary. In the decay-to-zero case, let \(V\) be a stationary solution satisfying
$$V''+V(1-V)=0,$$
$$V(0)=0,$$
$$
V(x)\to 0$$ and $$V'(x)\to 0$$
as $x\to\infty.$
Multiplying the stationary equation by $V'$ gives the first integral
$$
\frac{1}{2}(V')^2+\frac{1}{2}V^2-\frac{1}{3}V^3=C.
$$
where $C$ is a constant.
The decay at infinity yields $C=0$. Evaluating at $x=0$ and using
$V(0)=0$, we obtain $V'(0)=0$. The  uniqueness of solutions for the corresponding
initial-value problem therefore gives
$$V(x)= 0,\quad x\in [0,\infty).$$ Hence, under the decay-to-zero condition, the autonomous homogeneous Dirichlet problem for the Fisher--KPP equation admits no nonzero $T$-periodic solution. By contrast, this observation does not exclude the nonzero stationary Dirichlet profile satisfying
$$
V_D(0)=0,$$ with $$V_D(x)\to 1.
$$}
\end{remark}

As for the initial-boundary value problem for the autonomous Fisher-KPP equation with Robin condition, we first give the following lemma on the local asymptotic equivalence of Robin solutions.  
\begin{lemma}
\label{lem:robin-equivalence}
Let $u_R$ and $v_R$ be two solutions of the initial-boundary value problem for the autonomous Fisher--KPP equation
on $[0,\infty)$ with Robin boundary condition. Assume that for every $R>0$,
$$
\lim_{t\to\infty}
\sup_{x\in[0,R]} |u_R(x,t)-V_R(x)|=0,
$$
and
$$
\lim_{t\to\infty}
\sup_{x\in[0,R]} |v_R(x,t)-V_R(x)|=0.
$$
Then, for every $R>0$,
$$
\lim_{t\to\infty}
\sup_{x\in[0,R]} |u_R(x,t)-v_R(x,t)|=0.
$$
\end{lemma}
    
\textbf{Proof: } This follows directly from the triangle inequality. $\blacksquare$

The local uniform convergence of the Neumann solution is the content of the following lemma..     
\begin{lemma}
\label{lem:neumann-convergence}
Let $u_N$ be the classical solution of
$$(u_N)_t=(u_N)_{xx}+u_N(1-u_N),
    \qquad x>0,\ t>0,$$
with Neumann boundary condition 
$$(u_N)_x(0,t)=0,\quad t>0.$$
and initial condition
$$u_N(x,0)=\phi(x), \quad x>0.$$
Let $C_b^2[0,\infty)$  be  the function space defined by
$$C^2_b[0,\infty)=\left\{\phi\in C^2[0,\infty):\sup_{x\in [0,\infty)}|\phi(x)|<\infty\right\}.$$ Let $\phi\in C_b^2[0,\infty)$ be such that $$\phi'(0)=0,$$
$$0\le \phi(x)\le 1,\quad  x\in [0,\infty),$$
and $\phi$ is not identically zero on $[0,\infty).$
Then
$$0\le u_N(x,t)\le 1,
    \qquad x\ge0,\ t\ge0,$$
and for every $R>0$,
$$\lim_{t\to\infty}
    \sup_{x\in[0,R]} |u_N(x,t)-1|=0.$$
Equivalently,
$$ u_N(\cdot,t)\to 1$$ uniformly on every compact subset of $[0,\infty).$
\end{lemma}

\textbf{Proof:}
The bound
$$0\leq u_N(x,t)\leq 1,\quad x\geq 0,\, t\geq 0,$$
follows from the parabolic maximum principle; see, for example, \cite{EK}. Indeed, $0$ and $1$ are respectively a sub-solution and a super-solution of the Fisher--KPP equation.

It remains to prove the convergence to $1$ uniformly on every compact subset of $[0,\infty)$. Fix $R>0$. Since
$\phi$ is not identically zero on 
$[0,\infty)$, there exists an interval on which $\phi$ is
positive. Choose a compact interval
$I$ of $(0,\infty)$
such that $$\phi(x)>0,\quad x\in I.$$ Choose a number $L$ such that
$$L>\max\left\{R,\frac{\pi}{2}\right\}$$
and $$I\subset(0,L),$$ Now, choose a nonzero function $\psi$ in $C_c^\infty(I)$ such that
$$0\le\psi\le\phi.
$$
Consider the auxiliary problem on the finite interval $[0,L]$ given by
$$\begin{cases}
v_t=v_{xx}+v(1-v), & 0<x<L,\ t>0,\\
v_x(0,t)=0, & t>0,\\
v(L,t)=0, & t>0,\\
v(x,0)=\psi(x), & 0\le x\le L.
\end{cases}
$$
By the comparison principle\cite{EK}, since
$$v(x,0)=\psi(x)\le \phi(x)=u_N(x,0),\quad x\in [0,L],$$
and since
$$v(L,t)=0\le u_N(L,t),\quad t>0,$$
while both functions satisfy the Neumann boundary condition at $x=0$, we obtain
$$0\le v(x,t)\le u_N(x,t),\qquad 0\le x\le L,\ t\ge 0.$$
The principal eigenvalue of the mixed Neumann-Dirichlet problem is
$$
\lambda_1(L)=\left(\frac{\pi}{2L}\right)^2<1=f'(0).
$$
By the standard principal-eigenvalue criterion for positive
stationary solutions of logistic reaction-diffusion equations
on bounded intervals (see \cite{AB}), the inequality
$$
\lambda_1(L)<1=f'(0)
$$
implies that the finite-interval problem admits a unique
positive stationary solution $V_L$ satisfying
$$
\begin{cases}
    V_L''+V_L(1-V_L)=0, & 0<x<L,\\
    V_L'(0)=0,\\
    V_L(L)=0,
\end{cases}
$$
and the solution $v$ converges uniformly to $V_L$ on $[0,L]$, i.e.,
$$\lim_{t\to\infty}\|v(\cdot,t)-V_L\|_{L^\infty(0,L)}=0.$$
Hence,
$$ \mathop{\underline{\lim}}_{t\to \infty} u_N(x,t)\ge V_L(x),
    \qquad 0\le x\le L.$$

By the standard large-domain limit for positive logistic
stationary solutions (see \cite{AB}), the profiles $V_L$ converge
to $1$ locally uniformly on $[0,\infty)$ as $L\to\infty$. Therefore, for the fixed $R>0$,
we may choose $L$ sufficiently large such that
$$V_L(x)\ge 1-\varepsilon,
    \qquad 0\le x\le R.$$

    Since $v(\cdot,t)\to V_L$ uniformly on $[0,L]$, and since
$V_L(x)\ge 1-\varepsilon$ on $[0,R]$, we have, for all sufficiently large $t$,
$$
v(x,t)\ge 1-2\varepsilon,\qquad 0\le x\le R.
$$
Since $v\le u_N$, it follows that
$$
u_N(x,t)\ge 1-2\varepsilon,\qquad 0\le x\le R
$$
for all sufficiently large $t$.
Since $0\le u_N\le1$, for all sufficiently large $t$,
$$
\sup_{x\in[0,R]}|u_N(x,t)-1|
 =1-\inf_{x\in[0,R]}u_N(x,t)
 \le2\varepsilon.
$$
Since $\varepsilon>0$ is arbitrary, the uniform convergence on every compact subset of $[0,\infty)$
follows.  ${\blacksquare}$
  
\section{Error Analysis and Asymptotic Behavior}
We first consider the autonomous logistic equation
$$
u_t=u_{xx}+u(1-u), \quad x\geq 0,\,t\geq 0,
$$
and derive the stationary profile and long-time comparison estimates.
After establishing the exponential stability of the $C_0$-semigroup corresponding to
the linearized heat equations, we turn to small time-periodic Neumann and
Robin boundary forcing in
Theorem~\ref{thm:periodic-existence} and
Proposition~\ref{prop:periodic-expansion}.

The stationary profile estimates are obtained by using energy
identities and comparison arguments for the logistic stationary profiles.
The far-field condition in this section is
$$
V(x)\longrightarrow1
\qquad\text{as }x\to\infty,
$$
rather than $V(x)\to0$. The corresponding energy level contains a
bounded nonnegative branch approaching the equilibrium
$$(V,V')=(1,0).$$

\begin{theorem}
\label{thm:stationary-profiles}
 Let $V_D$,$V_R$ and $ V_N$ be  the stationary profiles satisfying
$$
V''+V(1-V)=0,\qquad x>0,
$$
with boundary conditions
$$
V_D(0)=0,$$  $$V_R'(0)=V_R(0)$$ and $$ V_N'(0)=0.
$$
Assume that
$$
V_D(x)\to 1,$$ $$V_R(x)\to 1$$ and $$V_N(x)\to 1$$
as $x\to \infty.$ Moreover, assume that
$$V_D'(x)\to 0,$$ $$V_R'(x)\to 0$$ and $$V_N'(x)\to 0$$
as $x\to\infty$.
In addition, suppose that
$$
0\leq V_D(x)\leq 1,$$
$$0<V_R(x)\leq 1$$ and
$$0<V_N(x)\leq1 .$$
Then $$V_N(x)=1\quad x\in [0,\infty).$$ and  for every $\delta\in(0,1)$,
there exists a positive constant $C_\delta$ such that
$$|V_D(x)-V_R(x)|+|V_N(x)-V_R(x)|
\leq C_\delta e^{-(1-\delta)x}$$ as $x\to \infty.$

\end{theorem}

\textbf{Proof:}
We first identify the Neumann stationary profile. Since $V_N$ satisfies
$$V_N''+V_N(1-V_N)=0,$$
$$ V_N'(0)=0$$
and  $$V_N(x)\to1$$ as $x\to\infty,$
we multiply the equation by $V_N'$. This gives
$$V_N''V_N' + V_N(1-V_N)V_N'=0.$$
Hence
$$\frac{d}{dx}
    \left[ \frac12 (V_N')^2+\frac12 V_N^2-\frac13 V_N^3\right]=0.$$
Therefore
$$\frac12 (V_N')^2+\frac12 V_N^2-\frac13 V_N^3=C$$
for some constant $C$. Taking $x\rightarrow\infty$, using
$V_N(x)\rightarrow 1$ and $V_N'(x)\to0$, we obtain
$$C=\frac12-\frac13=\frac16.$$
Since $V_N'(0)=0$, we get
$$\frac12 V_N(0)^2-\frac13 V_N(0)^3=\frac16.$$
Equivalently,
$$2V_N(0)^3-3V_N(0)^2+1=0.$$
Factoring the polynomial gives
$$ (V_N(0)-1)^2(2V_N(0)+1)=0.$$
Since $V_N>0$, it follows that
$$V_N(0)=1.$$
Together with $V_N'(0)=0$, this gives the initial data
$$
V_N(0)=1$$ and $$V_N'(0)=0.
$$
The constant function $$V(x)=1, \quad x\in [0,\infty)$$ satisfies the same initial value problem
$$\left\{\begin{array}{lll}
V''+V(1-V)=0,\\
V(0)=1,\\
V'(0)=0.\end{array}\right.$$
By the uniqueness theorem for ordinary differential equations, we conclude that
$$
V_N(x)=1,\quad x\in [0,\infty).
$$
The same first integral holds for $V_D$ and $V_R$, because they
satisfy the same far-field conditions. So,
$$
\frac12(V')^2+\frac12V^2-\frac13V^3=\frac16.
$$
Equivalently,
$$
(V')^2=\frac{(1-V)^2(2V+1)}{3}.
$$
For the Dirichlet profile, evaluation at \(x=0\) gives
$$
\bigl(V_D'(0)\bigr)^2=\frac13.
$$
Since \(V_D(0)=0\) and \(V_D\ge0\), it follows that
$$
V_D'(0)=\frac1{\sqrt3}>0.
$$
For the Robin profile,
$$
V_R'(0)=V_R(0)>0.
$$
Therefore both non-constant profiles lie on the increasing branch
$$
V'=(1-V)\sqrt{\frac{2V+1}{3}},
$$
and approach $1$ monotonically. Thus, the hypotheses of the
theorem select a bounded nonnegative branch of the separatrix; they do
not select an unbounded orbit or a spatially periodic sign-changing
orbit.
We can now derive the exponential tail estimate for the Robin stationary
profile. Define the function $z_R$ on $[0,\infty)$ by
$$z_R(x)=1-V_R(x),\quad x\in [0,\infty).$$
Since $$V_R''+V_R(1-V_R)=0,$$ we obtain
$$z_R''=z_R-z_R^2=z_R(1-z_R).$$
By assumption, $$V_R(x)\rightarrow1$$ as $x\to \infty.$ Therefore
$$z_R(x)\to0$$ as $x\to\infty.$
Fix $\delta\in(0,1)$. Choose $X_\delta>0$ sufficiently large so that
for all $x\ge X_\delta$,
$$0\le z_R(x)\le 1-(1-\delta)^2.$$
Then
$$z_R''(x)=z_R(x)(1-z_R(x))
    \ge (1-\delta)^2 z_R(x),
    \quad x\ge X_\delta.$$
Let $W$ be the function on $[X_\delta,\infty)$ be defined by
$$W(x)=C e^{-(1-\delta)(x-X_\delta)},\quad x \in [X_\delta,\infty),$$ where $C$ is a positive constant to be chosen
sufficiently large so that
$$z_R(X_\delta)\le W(X_\delta).$$
We claim that
$$ z_R(x)\le W(x),
    \quad x\ge X_\delta.$$
Suppose, for contradiction, that this is false. Define the function $Z$ on $[X_\delta,\infty)$ by
$$Z(x)=z_R(x)-W(x),\quad x\in [X_\delta,\infty).$$
Then $Z(X_\delta)\le0$. Since $$z_R(x)\to 0$$ and $$W(x)\to 0$$ 
as $x\to\infty$, we have
$$\lim_{x\to\infty} Z(x)=0.$$
If $Z$ is positive somewhere on $[X_\delta,\infty)$, then $Z$
attains a positive interior maximum at some point
$x_*>X_\delta$. At this point,
$$Z(x_*)>0,\qquad Z''(x_*)\leq0.$$
However,
$$ Z''(x_*)
    =z_R''(x_*)-W''(x_*)
    \ge (1-\delta)^2 z_R(x_*)-(1-\delta)^2 W(x_*)
    =(1-\delta)^2 Z(x_*)>0,$$
which is a contradiction. Hence
$$z_R(x)\le W(x),
    \qquad x\geq X_\delta.$$
Therefore
$$ 1-V_R(x)=z_R(x)\le C_\delta e^{-(1-\delta)x},\quad x\in [X_\delta,\infty).$$
Since $$V_N(x)=1, \quad x\in [X_\delta,\infty),$$ it follows that
$$|V_N(x)-V_R(x)|
    =1-V_R(x)
    \le C_\delta e^{-(1-\delta)x},\quad x\in [X_\delta,\infty).$$

The same argument applies to the Dirichlet stationary profile. Indeed,
if
$$ z_D(x)=1-V_D(x),\quad x\in [0,\infty),$$
then
$$ z_D''=z_D-z_D^2=z_D(1-z_D)$$
and $$z_D(x)\rightarrow0$$ as $x\rightarrow\infty$. Thus, for every
$\delta\in(0,1)$,
$$1-V_D(x)\le C_\delta e^{-(1-\delta)x},\quad x\in [X_\delta,\infty).$$
Increasing $C_\delta$ if necessary, this estimate holds for all $x\ge 0$.
Consequently,
$$ |V_D(x)-V_R(x)|
    \le |1-V_D(x)|+|1-V_R(x)|
    \le C_\delta e^{-(1-\delta)x}$$ as $x\to \infty.$
This proves the stationary error estimate. $\blacksquare$ 

\begin{remark}
{\rm The existence and uniqueness of the nontrivial Dirichlet stationary
profile have already been established in Corollary~6.3 of Yi and Zhao
\cite{YZ}. For completeness, the calculation below describes an
elementary phase-plane characterization of the same profile in the
normalization used here. For the Dirichlet profile, the first integral gives
$$
V_D(0)=0$$
and
$$V_D'(0)=\frac1{\sqrt3}.
$$
The increasing branch is determined by
$$
V'
=
(1-V)\sqrt{\frac{2V+1}{3}}.
$$
The technique of separation of variables shows that this solution is defined for all
$x\ge0$ and satisfies $$V_D(x)\to1$$ as $x\to \infty.$
For the Robin profile, let $a_R=V_R(0)$. Since
$$V_R'(0)=V_R(0)=a_R,$$ the first integral gives
$$
2a_R^3-6a_R^2+1=0.
$$
The function
$$
g(a)=2a^3-6a^2+1
$$
is strictly decreasing on $(0,1)$, with $g(0)>0$ and $g(1)<0$.
Hence there exists a unique $a_R\in(0,1)$. The corresponding
increasing solution of
$$
V'
=
(1-V)\sqrt{\frac{2V+1}{3}}
$$
is defined on $[0,\infty)$ and converges to $1$.}
\end{remark}
\begin{remark}
\label{rem:stationary-orbits}
{\rm Theorem~\ref{thm:stationary-profiles} does not concern the stationary
problem with the far-field condition $$V(x)\to0$$ as $x\to \infty$. Statements about
unbounded or spatially periodic sign-changing solutions refer to
different phase-plane orbits and do not contradict the profiles used in this theorem.} 
\end{remark}

The following theorem gives  the local uniform convergence of the difference of the Neumann problem and the Robin problem.

\begin{theorem}
\label{thm:NR-local-decomposition}
Let $u_N$ and $u_R$ be classical solutions of
$$
u_t=u_{xx}+u(1-u),\qquad x>0,\quad t>0,
$$
with, respectively, boundary conditions
$$
(u_N)_x(0,t)=0,\quad t>0,$$
and
$$(u_R)_x(0,t)=u_R(0,t),\quad t>0.
$$
Let $\phi_N$ and $\phi_R$ be, respectively, the initial data in $C_b^2[0,\infty)$ such that
$$
\phi_N'(0)=0,$$
$$
\phi_R'(0)=\phi_R(0),$$
$$0\leq \phi_N(x)\leq 1,$$
$$0\leq \phi_R(x)\leq 1$$
for all $x\in [0,\infty)$
and
$\phi_N$ is not identically zero on $[0,\infty)$. Let $V_R$ be the positive Robin stationary profile satisfying
$$
V_R''+V_R(1-V_R)=0,$$
$$V_R'(0)=V_R(0)$$
and
$$V_R(x)\to1$$
as $x\to\infty.$
Assume explicitly that $u_R(\cdot,t)\to V_R$ locally uniformly as
$t\to\infty$.
For every $R>0$, define functions $E_N^R$ and $E_R^R$ on $[0,\infty)$ by
$$
E_N^R(t)
=
\sup_{x\in[0,R]}|u_N(x,t)-1|$$
and
$$E_R^R(t)
=
\sup_{x\in[0,R]}|u_R(x,t)-V_R(x)|.
$$
Then
$$
\sup_{x\in[0,R]}
\left|
u_N(x,t)-u_R(x,t)-\bigl(1-V_R(x)\bigr)
\right|
\le E_N^R(t)+E_R^R(t).
$$
Consequently,
$$
u_N(\cdot,t)-u_R(\cdot,t)
\longrightarrow 1-V_R
$$
uniformly on every compact interval as $t\to\infty$.
For initial data close to the corresponding stationary profiles, a
quantitative exponential version of this comparison is given in
{\rm Corollary~\ref{cor:quantitative-autonomous-NR}} below.
\end{theorem}

\textbf{Proof:}
By Theorem~3.1, the Neumann stationary profile is $V_N\equiv1$.
For every $x\in[0,R]$,
$$
\begin{aligned}
&\left|
u_N(x,t)-u_R(x,t)-\bigl(1-V_R(x)\bigr)
\right|\\
&\qquad
=
\left|
\bigl(u_N(x,t)-1\bigr)
-\bigl(u_R(x,t)-V_R(x)\bigr)
\right|\\
&\qquad
\le
|u_N(x,t)-1|
+
|u_R(x,t)-V_R(x)|.
\end{aligned}
$$
Taking the supremum over $x\in[0,R]$ gives
$$
\sup_{x\in[0,R]}
\left|
u_N(x,t)-u_R(x,t)-\bigl(1-V_R(x)\bigr)
\right|
\le E_N^R(t)+E_R^R(t).
$$

Lemma~2.2 gives
$$
E_N^R(t)\to0,
$$
and the assumed local uniform convergence of the Robin solution gives
$$
E_R^R(t)\to0.
$$
The claimed local uniform convergence follows $\blacksquare$

We can also give as a corollary the far-field difference between the Neumannn solution and the Robin solution.

\begin{corollary}
\label{cor:NR-far-field}
Under the hypotheses of
{\rm Theorem \ref{thm:NR-local-decomposition}}, for every
\(\delta\in(0,1)\), there exists a positive constant $C_\delta>0$ such that for all positive numbers $R$,
$$
|u_N(x,t)-u_R(x,t)|
\le
E_N^R(t)+E_R^R(t)
+C_\delta e^{-(1-\delta)x},\quad x\in [0, R],\, t\in [0,\infty).$$
Consequently,
$$
\lim_{x\to\infty}
\overline{\lim_{t\to\infty}}
|u_N(x,t)-u_R(x,t)|=0.
$$
\end{corollary}

\textbf{Proof:}
Theorem~3.1 gives
\[
|1-V_R(x)|
\le C_\delta e^{-(1-\delta)x},
\qquad x\ge0.
\]
Combining this estimate with
Theorem~\ref{thm:NR-local-decomposition}, we obtain
\[
\begin{aligned}
|u_N(x,t)-u_R(x,t)|
&\le
\left|
u_N(x,t)-u_R(x,t)-\bigl(1-V_R(x)\bigr)
\right|
+|1-V_R(x)|\\
&\le
E_N^R(t)+E_R^R(t)
+C_\delta e^{-(1-\delta)x}.
\end{aligned}
\]
for all $x\in [0,R]$ and $t\in [0,\infty).$
%In particular, for \(0\le A\le R\),
%\[
%\sup_{x\in[A,R]}
%|u_N(x,t)-u_R(x,t)|
%E_N^R(t)+E_R^R(t)
%+C_\delta e^{-(1-\delta)A}.
%\]
Let \(t\to\infty\). Then for every fixed $x\in [0,R]$,
$$\overline{\lim_{t\to\infty}}
|u_N(x,t)-u_R(x,t)|
\le C_\delta e^{-(1-\delta)x}.
$$
Finally, letting \(x\to\infty\) proves the result. $\blacksquare$

$C_0$-semigroups of bounded linear operators on Banach spaces are versatile tools for the solutions of partial differential equations and pseudo-differential equations. See, for instance, \cite{DA2,IW}. The following proposition
is another manifestation of the same theme.
\begin{proposition}
\label{prop:linearized-stability}
Let $X$ be the Banach space $C_0[0,\infty)$ equipped with the supremum norm. Let $B\in\{B_N,B_R\}$, where
$$
B_N\phi=\phi'(0),\quad \phi\in X,$$
and
$$B_R\phi=\phi'(0)-\phi(0),\quad\phi \in X.
$$
Let $V_N$ and $V_R$ be the stationary profiles corresponding to the Fisher--KPP equation with respect to the Neumann boundary condition and the Robin condition respectively.  
Define the linear operator $L_B:{\mathcal(D)}_B)\to X$ by first letting ${\mathcal{D}}(L_B)$  be the set of all functions $\phi$ in $X$ such that $\phi$  and $\phi'$ are locally absolutely continuous on $[0,\infty)$, $\phi''\in X$ and
$B\phi=0.$
For all $\phi\in X$, we define $L_B\phi$ by
$$L_B\phi=\phi''+(1-2V_B)\phi.$$
Let  $\{S_B(t):t\geq0\}$ be the positive $C_0$-semigroup generated by $L_B$ on
$X$. Moreover, let $m_B$ be the number given by
$$
m_B:=\inf_{x\geq 0}V_B(x).
$$
%For $B=B_N$, we have $m_N=1$. For $B=B_R$,
%Theorem~3.1 shows that $V_R$ is increasing, and therefore
%$$
%m_R=V_R(0)>0.
%$$
%Consequently,
%$$
Then
$$\|S_B(t)\|_{B(X)}
\leq \frac{1}{m_B}e^{-m_Bt},
\quad t\geq 0,
$$
where $\|\,\|_{B(X)}$ is the norm in the Banach space of all bounded linear operators on $X$.
In particular, $S_B(t)$ is exponentially stable.
\end{proposition}
\textbf{Proof}
The operator $L_B$ is the differential operator $\frac{d^2}{dx^2}$ perturbed by the multiplication operator given by the bounded function $1-2V_B$ on $X$ with dense domain ${\mathcal{D}}(L_B)$. Since $\frac{d^2}{dx^2}$ is the generator of a positive $C_0$-semigroup of bounded linear operators on $X$, it follows from Theorem 3.1 of the book \cite{Da} and the Trotter Product formula therein that
$\frac{d^2}{dx^2}+(1-2V_B)$ is the generator of a positive $C_0$-semigroup $\{S_B(t):t\geq 0\}$ of bounded linear operators on $X$.
Let $\phi\in X$. Then we define $u:[0,\infty)\to X$ by
$$
u(t)=S_B(t)\phi,\quad t\in [0,\infty).
$$
Since $V_B\ge m_B>0$, it follows that
$$
|\phi(x)|
\le \frac{\|\phi\|_X}{m_B}V_B(x), \quad x\in [0,\infty).
$$
Let $Z$ be the function on $[0,\infty)\times [0,\infty)$  given by
$$
Z(x,t)
=\frac{\|\phi\|_X}{m_B}e^{-m_Bt}V_B(x),\quad x\geq 0,\, t\geq 0.
$$
Since
$$
V_B''+V_B(1-V_B)=0,
$$
we have
$$
L_BV_B=-V_B^2.
$$
Consequently,
$$
Z_t(x,t)-(L_BZ)(x,t)
=
\frac{\|\phi\|_X}{m_B}e^{-m_Bt}
V_B(V_B-m_B), \quad t\geq 0.$$
Moreover,
$$
BZ(\cdot,t)=0, \quad t\geq 0,$$
and
$$-Z(\cdot,0)\leq \phi\leq Z(\cdot,0).
$$
The estimate is trivial when $\phi=0$. So, assume that
$\|\phi\|_X>0$. We first take $\phi\in D(L_B)$, so that
$$
u(t)=S_B(t)\phi,\quad t\geq 0,
$$
is a classical solution for positive time.
Fix a positive number $\tau$. Since the map
$$
t\longmapsto u(t)
$$
is continuous from $[0,\tau]$ into $X$, the set
$
\{u(t):0\leq t\leq \tau\}
$
is compact in $X$. 
Since $u(t)\in X,$ it follows that for all positive numbers $L$, 
$$\lim_{L\to \infty}\sup_{0\leq t\leq \tau}\sup_{x\geq L}|u(t)x|=0.$$
as $L\to \infty.$
On the other hand,
$$
\inf_{0\leq t\leq \tau} Z(x,t)
=
\frac{\|\varphi\|_X}{m_B}
e^{-m_B\tau}V_B(x)
\longrightarrow
\frac{\|\varphi\|_X}{m_B}e^{-m_B\tau}>0
$$
as $x\to\infty$. Therefore for all sufficiently large positive numbers $L$,
$$
-Z(L,t)\leq u(t)L)\leq Z(L,t),
\qquad 0\leq t\leq \tau.
$$
Applying the parabolic comparison principle to $Z-u$ and $Z+u$
on $[0,L]\times[0,\tau]$, and then letting $L\to\infty$, we get
$$
|S_B(t)\varphi(x)|\leq Z(x,t),
\qquad x\geq 0,\quad 0\leq t\leq \tau.
$$
Since $\tau$ is arbitrary, this estimate holds for every $t\in [0,\infty)$.
Finally, the estimate extends from $\mathcal{D}(L_B)$ to all off $X$ by density
and the strong continuity of $\{S_B(t):t\geq 0\}$. 
Since $0<V_B\le1$, it follows that 
$$
\|S_B(t)\phi\|_X
\le\frac{1}{m_B}e^{-m_Bt}\|\phi\|_X,\quad t\geq 0.
$$
So,
$$\|S_B(t)\|_{B(X)}\leq \frac{1}{m_B}e^{-m_Bt},\quad t\geq 0.$$
This completes the proof. $\blacksquare$

\section{Periodic Boundary Forcing of the Autonomous Fisher--KPP Equation}
Mild solutions are well known and are used in connection with solving partial differential equations by $C_0$-semigroups of bounded linear operators on Banach spaces. See, for instance, \cite{DK,GW}. 

A periodic boundary forcing is to introduce a time-periodic term $p$ to the boundary operator $B$. More precisely, we have 
$$Bu(0,t)=BV_B(0)+\varepsilon p(t),\quad t\geq 0.$$

A lifted mild solution $\{P_B^\varepsilon(t):t\geq 0\}$ 
is given  by
$$
P_B^\varepsilon(t)
=
V_B+\varepsilon p(t)\rho_B+z_B,
$$
where $\rho_B\in C_c^2[0,\infty)$, $B\rho_B=1$ and $z_B$ is the $X$-valued mild solution of
the corresponding problem with homogeneous initial condition.
Thus, $P_B^\varepsilon(t)$ is considered in the affine space
$V_B+X$; only the perturbation
$$
P_B^\varepsilon(t)-V_B
$$
is required to belong to $X$. In particular, the steady state $V_B$
itself need not belong to $C_0[0,\infty)$. For example, one may choose $\rho_B$ so that
$$
\rho'_N(0)=1
$$
in the Neumann case, and
$$
\rho'_R(0)-\rho_R(0)=1
$$
in the Robin case.

We now give the existence, uniqueness 
and a first-order bound for lifted time-periodic mild solutions.

\begin{theorem}
\label{thm:periodic-existence}
For a positive number $T$, we let $C_{\rm per}^1[0,T]$ be the set of all $C^1$ and periodic functions with period $T$ on $[0,\infty)$.
%Let
%$$
%X=C_0([0,\infty))
%$$
%be equipped with the supremum norm. Let the boundary operator $B$
%be either
%$$
%B_Nu=u_x(0)
%$$
%or
%$$
%B_Ru=u_x(0)-u(0).
%$$
Consider
$$
u_t=u_{xx}+f(u),
\qquad x>0,\quad t\in\mathbb R,
$$
with periodically forced boundary condition
where $b_0$ is the constant given by
$$b_0=V_B(0),$$ $\varepsilon$ is a positive number
and 
$$p\in C_{\rm per}^1[0,T].$$ Assume that there is an open interval $I_B$ such that $$f\in C^2(I_B)$$ and $$\overline{V_B[0,\infty)}\subseteq I_B.$$ Let ${\mathcal{L}}_B:\mathcal{D}(L_B)\to X$ be the differential operator defined by
\begin{equation}\label{LBScript}
{\mathcal{L}}_B\phi=\phi''+f'(V_B)\phi, \quad \phi\in \mathcal{D}(L_B).\end{equation} Assume also that there exists a steady state \(V_B\) satisfying
$$
V_B''+f(V_B)=0,$$
and
$$
V_B(x)\to V_\infty$$ for some constant $V_\infty$
as $x\to\infty,$ 
%Assume further that there exists an open interval $I_B$
%such that
%$$
%\overline{V_B([0,\infty))}\subset I_B
%\qquad\text{and}\qquad
%f\in C^2(I_B).
%$$
%Define the linearized differential expression
%$$
%\mathscr L_B\phi
%$$:=
%\phi''+f'(V_B)\phi,
%$$
%and let $L_B$ be its homogeneous realization on $X$:
%$$
%L_B\phi=\mathscr L_B\phi,
%\qquad
%\phi\in D(L_B),
%$$
%where
%$$
%D(L_B)
%=
%\left\{
%\phi\in X:
%\begin{array}{l}
%\phi,\phi'\text{ are locally absolutely continuous},\\
%\phi''\in X,\quad B\phi=0
%\end{array}
%\right\}.
%$$
%Assume that $L_B$ generates a $C_0$-semigroup $S_B(t)$ satisfying
%$$
%\|S_B(t)\|_{\mathcal L(X)}
%\le M_Be^{-\gamma_Bt},
%%\qquad t\ge0,
%$$
%for some $M_B\ge1$ and $\gamma_B>0$.
Let $\rho_B\in C_c^2[0,\infty)$ be such that
$$
B\rho_B=1.
$$
Then there exist positive constants
$
\varepsilon$,
$r_B$
and
$C_B$
such that for every positive number $\varepsilon$ with $\varepsilon<\varepsilon_B$, there exists a
$T$-periodic lifted mild solution
$$
P_B^\varepsilon(t)
 =V_B+\varepsilon p(t)\rho_B+z_B^\varepsilon(t),
\quad t\geq 0,$$ with
$$z_B^\varepsilon\in C_{\mathrm{per}}([0,T];X).$$
It is unique among all lifted $T$-periodic mild solutions satisfying
the fixed-neighborhood condition
$$
\sup_{t\in[0,T]}
\|P_B^\varepsilon(t)-V_B\|_X
\le r_B.
$$
Moreover,
$$
\sup_{t\in[0,T]}
\|P_B^\varepsilon(t)-V_B\|_X
\le C_B\varepsilon.
$$
\end{theorem}
\textbf{Proof:}
 We give the proof for $B\in\{B_N,B_R\}$. The Neumann and Robin cases differ only in the choice of the lifting function.
The role of the lifting function is to convert the nonhomogeneous boundary condition into a homogeneous one before applying the semigroup generated by $L_B$.
Choose a function

$$
\rho_B\in C_c^2([0,\infty))
$$

such that

$$
B\rho_B=1.
$$

Set

$$
w_B(x,t)=P_B^\varepsilon(t)\phi(x)-V_B(x), \quad x\geq 0,\, t\geq  0.
$$

Then $w_B$ satisfies

$$
(w_B)_t=(w_B)_{xx}+f(V_B+w_B)-f(V_B),\quad x\geq 0,\,t\geq0.
$$

with boundary condition

$$
Bw_B(\cdot,t)=\varepsilon p(t),\quad t\geq 0.
$$

We remove the nonhomogeneous boundary condition by writing

$$
w_B(x,t)=\varepsilon p(t)\rho_B(x)+z_B(x,t),\quad x\geq 0,\,t\geq 0.
$$

Then

$$
Bz_B(\cdot,t)=0,\quad t\geq 0.
$$

Using the definitions of $\mathcal L_B$ and $L_B$, we see that
$${\mathcal{L}_B}=L_B$$ and we
obtain
$$
(z_B^\varepsilon)_t
=
L_Bz_B^\varepsilon
+\varepsilon F_B(\cdot,t)
+N_B\bigl(\varepsilon p(t)\rho_B+z_B^\varepsilon\bigr),\quad t\geq 0,
$$
where
$$
F_B(x,t)
=
p(t)\mathcal L_B\rho_B(x)-p'(t)\rho_B(x),
$$
and
\begin{equation}\label{NB(q)}
N_B(q)
=
f(V_B+q)-f(V_B)-f'(V_B)q,\quad q\geq 0.
\end{equation}
Since $p\in C^1_{\mathrm{per}}[0,T]$ and $\rho_B\in C_c^2[0,\infty)$, we have

$$
F_B\in C_{\mathrm{per}}([0,T];X).
$$

Moreover, since $f\in C^2[0,\infty)$, there exist constants $K_B>0$ and $r_0>0$ such that whenever

$$
\|q\|_X\le r_0,
$$

we have

$$
\|N_B(q)\|_X\le K_B\|q\|_X^2,
$$

and whenever

$$
\|q_1\|_X\leq r_0$$ and $$\|q_2\|_X\le r_0,
$$

we have

$$
\|N_B(q_1)-N_B(q_2)\|_X
\le K_B(\|q_1\|_X+\|q_2\|_X)\|q_1-q_2\|_X.
$$

Let

$$
Y_B=C_{\mathrm{per}}([0,T];X)
$$

with norm

$$
\|z\|_{Y_B}=\sup_{t\in[0,T]}\|z(\cdot,t)\|_X,\quad z\in Y_B.
$$

For $z\in Y_B$, define
$$
(\mathcal T_B z)(t)
=
\int_{-\infty}^{t}
S_B(t-s)
\left[
\varepsilon F_B(\cdot,s)
+
N_B\big(\varepsilon p(s)\rho_B+z(\cdot,s)\big)
\right]\,ds,\quad t\geq 0.
$$

The integral is well-defined because $\{S_B(t):t\geq 0\}$ is exponentially stable. Also, since $F_B$, $p$, and $z$ are $T$-periodic, the function $\mathcal T_Bz$ is $T$-periodic.

So, if $z\in Y_B$ and
$$
\|\varepsilon p\rho_B+z\|_{Y_B}\le r_0,
$$
then
$$
\begin{aligned}
\|\mathcal T_B z\|_{Y_B}
&\le
\sup_{t\in[0,T]}
\int_{-\infty}^{t}
\|S_B(t-s)\|_{B(X)}
\left(
\varepsilon\|F_B(\cdot,t)\|_X
+
\|N_B(\varepsilon p(s)\rho_B+z(\cdot,s))\|_X
\right)\,ds  \\
&\le
\sup_{t\in[0,T]}
\int_{-\infty}^{t}
M_Be^{-m_B(t-s)}
\left(
\varepsilon\|F_B\|_{Y_B}
+
K_B(\varepsilon p\rho_B+z\|_{Y_B}^{2}
\right)\,ds  \\
&\le
\frac{M_B}{m_B}
\left(
\varepsilon\|F_B\|_{Y_B}
+
K_B\big(\varepsilon\|p\rho_B\|_{Y_B}+\|z\|_{Y_B}\big)^2
\right).
\end{aligned}
$$
Let
$$
A_B=\frac{M_B}{m_B},$$
$$F_B^*=\|F_B\|_{Y_B}$$ and
$$Q_B^*=\|p\rho_B\|_{Y_B}.
$$
Choose a fixed positive number $R_B$ so small that
$$
2R_B<r_0,$$
and
$$4A_BK_BR_B\le\frac12.
$$
Next, choose a positive number $\varepsilon_B$ sufficiently small such that whenever
$\varepsilon<\varepsilon_B$,
$$
\varepsilon Q_B^*\le R_B$$
$$A_B\varepsilon F_B^*\le\frac{R_B}{2}.
$$
Define the fixed closed ball
$$
\mathcal B_{R_B}
 =\left\{
 z\in Y_B:\|z\|_{Y_B}\le R_B
 \right\}.
$$
If $z\in\mathcal B_{R_B}$, then
$$
\|\varepsilon p\rho_B+z\|_{Y_B}
 \le |\varepsilon|Q_B^*+R_B
 \le2R_B<r_0.
$$
Therefore
$$
\begin{aligned}
\|\mathcal T_Bz\|_{Y_B}
&\le
A_B\varepsilon F_B^*
 +A_BK_B
 \bigl(\varepsilon Q_B^*+R_B\bigr)^2\\
&\le
\frac{R_B}{2}+4A_BK_BR_B^2\\
&\le R_B.
\end{aligned}
$$
Thus, $\mathcal T_B$ maps the ball
$\mathcal B_{R_B}$ into itself. Next, let $z_1$ and $z_2$ be elements in $\mathcal B_{R_B}$ and set
$$
q_j=\varepsilon p\rho_B+z_j,
\qquad j=1,2.
$$
Then
$$
\|q_j\|_{Y_B}\le2R_B<r_0.
$$
Using the Lipschitz continuity of $N_B$, we obtain $$
\begin{aligned}
\|\mathcal T_B z_1-\mathcal T_B z_2\|_{Y_B}
&\leq
A_BK_B\bigl(\|q_1\|_{Y_B}+\|q_2\|_{Y_B}\bigr)
\|z_1-z_2\|_{Y_B} \\
&\leq
4A_BK_BR_B\|z_1-z_2\|_{Y_B} \\
&\leq
\frac12\|z_1-z_2\|_{Y_B}.
\end{aligned}
$$
Hence $\mathcal T_B$ is a contraction on $\mathcal B_{R_B}$.
By the contraction mapping theorem, it has a unique fixed point
$$
z_B^\varepsilon\in\mathcal B_{R_B}.
$$
Therefore
$$
P_B^\varepsilon(t)
=
V_B+\varepsilon p(t)\rho_B+z_B^\varepsilon(t),\quad t\geq 0.
$$
is a $T$-periodic lifted mild solution.
Now, we derive an $O(\varepsilon)$ estimate for the fixed point.
But
$$
\|z_B^\varepsilon\|_{Y_B}
\le
A_B\varepsilon F_B^*
+A_BK_B
\left(
\varepsilon Q_B^*
+\|z_B^\varepsilon\|_{Y_B}
\right)^2.
$$
Using
$$
(a+b)^2\le2a^2+2b^2
$$
and
$$
\|z_B^\varepsilon\|_{Y_B}^2
\le
R_B\|z_B^\varepsilon\|_{Y_B},
$$
we obtain
$$
\begin{aligned}
\left(1-2A_BK_BR_B\right)
\|z_B^\varepsilon\|_{Y_B}
&\le
A_B\varepsilon F_B^*
+2A_BK_B\varepsilon^2(Q_B^*)^2.
\end{aligned}
$$
Since
$$
2A_BK_BR_B\le\frac14,
$$
after decreasing $\varepsilon_B$, if necessary, so that
$\varepsilon_B\le1$, there exists a positive constant $C_B'$,
independent of $\varepsilon$, such that
$$
\|z_B^\varepsilon\|_{Y_B}
\le C_B'\varepsilon.
$$
Consequently,
$$
\begin{aligned}
\|P_B^\varepsilon-V_B\|_{Y_B}
&\le
\varepsilon Q_B^*
+\|z_B^\varepsilon\|_{Y_B}\\
&\le
(Q_B^*+C_B')\varepsilon.
\end{aligned}
$$
Thus, the asserted first-order estimate holds with

$$
C_B=Q_B^*+C_B'.
$$

It remains to prove uniqueness in a fixed neighborhood. Let $r_B$ be the positive number given by
$$
r_B=\frac{R_B}{2}.$$
Decreasing $\varepsilon_B$ further so that
$$
C_B\varepsilon\le r_B$$
and
$$\varepsilon Q_B^*\le\frac{R_B}{2}
$$
whenever $\varepsilon<\varepsilon_B$, we define $\tilde{z_B}$
for  any lifted $T$-periodic mild solution $\widetilde{P_B^\varepsilon}$
satisfying
$$
\|\widetilde P_B^\varepsilon-V_B\|_{Y_B}\le r_B,
$$
by
$$
\widetilde z_B
 =\widetilde P_B^\varepsilon
   -V_B-\varepsilon p\rho_B.
$$
Then
$$
\|\widetilde z_B\|_{Y_B}
\le
r_B+|\varepsilon|Q_B^*
\le R_B.
$$
Hence $\widetilde z_B\in\mathcal B_{R_B}$. For every positive integer $n$, periodicity and the variation-of-constants
give
$$
\begin{aligned}
\widetilde z_B(t)
&=S_B(nT)\widetilde z_B(t-nT)\\
&\quad+
\int_{t-nT}^{t}
S_B(t-s)
\left[
\varepsilon F_B(s)
+N_B\bigl(
\varepsilon p(s)\rho_B+\widetilde z_B(s)
\bigr)
\right]\,ds,\quad t\geq 0.
\end{aligned}
$$
Since $\widetilde z_B$ is bounded and $S_B(t)$ is exponentially
stable,
$$
S_B(nT)\widetilde z_B(t-nT)\longrightarrow0
\qquad\text{as }n\to\infty.
$$
Therefore
$$
\widetilde z_B=\mathcal T_B\widetilde z_B.
$$
Thus $\widetilde z_B$ is a fixed point of $\mathcal T_B$ in
$\mathcal B_{R_B}$. By uniqueness of the fixed point in this ball,
$$
\widetilde z_B=z_B^\varepsilon.
$$
Therefore
$$
\widetilde P_B^\varepsilon=P_B^\varepsilon.
$$
This proves uniqueness in the fixed neighborhood and completes the
proof. $\blacksquare$

From Theorem~\ref{thm:periodic-existence}, we  can give a result with a better estimate on  the remainder.

\begin{proposition}
\label{prop:periodic-expansion}
Under the hypotheses of
{\rm Theorem~\ref{thm:periodic-existence}}, let $\Phi_B$ be the unique
$T$-periodic lifted mild solution of
$$
(\Phi_B)_t
=
(\Phi_B)_{xx}+f'(V_B)\Phi_B,
\qquad x>0,\,t\geq 0,
$$
with boundary condition
$$
B\Phi_B(\cdot,t)=p(t),\quad t\geq 0.
$$
Then
$$
\sup_{t\in[0,T]}
\left\|
P_B^\varepsilon(t)-V_B-\varepsilon\Phi_B(t)
\right\|_X
\le C_B\varepsilon^2,
$$
after enlarging $C_B$ if necessary.
\end{proposition}

\textbf{Proof:}
Set
$$
F_B(t)
=
p(t) \mathscr L_B\rho_B-p'(t)\rho_B
$$
and define
$$
\eta_B(t)
=
\int_{-\infty}^t
S_B(t-s)F_B(s)\,ds,\quad t\geq 0.
$$
Exponential stability implies that this integral is well defined and
that $\eta_B$ is the unique $T$-periodic lifted mild solution of the
homogeneous linearized equation with forcing $F_B$.

Define
$$
\Phi_B(t)
=
p(t)\rho_B+\eta_B(t),\quad t\geq 0.
$$
Then $\Phi_B$ is the unique $T$-periodic lifted mild solution of
the stated linearized boundary problem.

Let
$$
R_B^\varepsilon(t)
=
P_B^\varepsilon(t)-V_B-\varepsilon\Phi_B(t),\quad t\geq 0 .
$$
Writing
$$
N_B(q)
=
f(V_B+q)-f(V_B)-f'(V_B)q,\quad q\geq 0.
$$
and subtracting the corresponding mild integral formulas gives

$$
R_B^\varepsilon(t)
=
\int_{-\infty}^t
S_B(t-s)
N_B\!\left(
\varepsilon\Phi_B(s)+R_B^\varepsilon(s)
\right)\,ds\quad t\geq 0.
$$

Since $f\in C^2[0,\infty)$, there exist positive constants $K_B0$ and $r_0$ such that
$$
\|N_B(q)\|_X
\le K_B\|q\|_X^2
$$
whenever \(\|q\|_X\le r_0\). By
Theorem~\ref{thm:periodic-existence},
$$
\sup_{t\in[0,T]}
\left\|
\varepsilon\Phi_B(t)+R_B^\varepsilon(t)
\right\|_X
=
\sup_{t\in[0,T]}
\|P_B^\varepsilon(t)-V_B\|_X
\le C_B\varepsilon.
$$
Therefore, if $\varepsilon$ is sufficiently small,
$$
\begin{aligned}
\sup_{t\in[0,T]}
\|R_B^\varepsilon(t)\|_X
&\le
\sup_{t\in[0,T]}
\int_{-\infty}^t
M_Be^{-\gamma_B(t-s)}
\left\|
N_B\!\left(
\varepsilon\Phi_B(s)+R_B^\varepsilon(s)
\right)
\right\|_X\,ds\\
&\le
\frac{M_BK_B}{\gamma_B}
\left(
\sup_{t\in[0,T]}
\|P_B^\varepsilon(t)-V_B\|_X
\right)^2\\
&\le
\frac{M_BK_BC_B^2}{\gamma_B}\,
\varepsilon^2.
\end{aligned}
$$
Enlarging \(C_B\) proves the result.$\blacksquare$

Since the exponential stability of
Theorem~\ref{thm:periodic-existence} follows from
Proposition~\ref{prop:linearized-stability}, we have the following corollary.

\begin{corollary}
\label{cor:logistic-periodic-forcing}
Let
$$
f(u)=u(1-u), \quad u\in [0,1],$$
$$b_0=0$$
and
$$B\in\{B_N,B_R\}.
$$
Let $V_B$ be the corresponding Neumann or Robin stationary
profile. 
Then the
conclusions of {\rm Theorem~\ref{thm:periodic-existence}} and
{\rm Proposition~\ref{prop:periodic-expansion}} hold for the logistic
Neumann and Robin problems.
\end{corollary}

\textbf{Proof:}
Proposition~\ref{prop:linearized-stability} gives
$$
\|S_B(t)\|_{B(X)}
\le
\frac{1}{m_B}e^{-m_Bt},
\qquad t\ge0.
$$
Thus, the conclusions hold with
$$
M_B=\frac{1}{m_B}$$
and
$$\gamma_B=m_B>0.
$$
The proof is complete. $\blacksquare$
\begin{remark}{\rm 
The Dirichlet problem with nonhomogeneous time-periodic boundary forcing requires a slightly
different semigroup formulation, since the homogeneous Dirichlet realization naturally acts on
the closed subspace $\{\phi\in C_0[0,\infty):\phi(0)=0\}$, rather than on the whole space
$C_0[0,\infty)$. For this reason, Theorem~\ref{thm:periodic-existence}  is stated only for the Neumann and Robin
boundary operators.}
\end{remark}

We give a result on the explicit profiles and the spatial decay under the Dirichlet, Neumann and Robin boundary conditions.
%\subsection{Sharp profiles and nonlinear stability}
\label{subsec:sharp-profiles-nonlinear-stability}
\begin{corollary}
\label{cor:explicit-sharp-profiles}
Under the hypotheses of
{\rm Theorem~\ref{thm:stationary-profiles}}, let $\sigma_D$ be the number defined by
$$
\sigma_D
 =2\tanh^{-1}\!\left(\frac{1}{\sqrt{3}}\right)
 =\log(2+\sqrt{3}).
$$
Then
$$
V_D(x)
 =1-\frac{3}{2}\sech^2\!\left(\frac{x+\sigma_D}{2}\right),\quad x\geq 0,$$
and
$$V_N(x)=1,\quad x\geq 0.
$$

Let $q_R\in\left(\frac{1}{\sqrt{3}},1\right)$ be the unique root of
$$
3q^3+3q^2-3q-1=0
$$
and define the number $\sigma_R$ by
$$
\sigma_R=2\tanh^{-1}(q_R).
$$
Then
$$
V_R(x)
 =1-\frac{3}{2}\sech^2\!\left(\frac{x+\sigma_R}{2}\right),\quad x\geq 0.
$$

Moreover,
$$
0\le V_D(x)<V_R(x)<V_N(x)=1,
\qquad x\ge0,
$$
and
\begin{eqnarray}
&&|V_D(x)-V_R(x)|+|V_N(x)-V_R(x)|\nonumber\\
&=&1-V_D(x)=\frac{3}{2}
 \sech^2\!\left(\frac{x+\sigma_D}{2}\right)
 \le6e^{-(x+\sigma_D)}\nonumber
 \end{eqnarray}
 for all $x\in[0,\infty).$
More precisely,
for $B\in\{D,R\}$,
$$
1-V_B(x)
 =6e^{-(x+\sigma_B)}+O(e^{-2x})$$
as $x\to\infty.$
\end{corollary}
\textbf{Proof:}
The first integral obtained in the proof of
Theorem~\ref{thm:stationary-profiles} is
$$
\frac12(V')^2+\frac12V^2-\frac13V^3=\frac16.
$$
Equivalently,
$$
(V')^2=\frac{(1-V)^2(2V+1)}{3}.
$$
A non-constant profile satisfying $0\le V<1$ and converging to $1$
lies on the increasing branch
$$
V'=(1-V)\sqrt{\frac{2V+1}{3}}.
$$

Let $q$ be the function given by
$$
q=\sqrt{\frac{2V+1}{3}}.
$$
Then
$$
V=\frac{3q^2-1}{2}.
$$
Since
$$
V'=(1-V)q
$$
and
$$
1-V=\frac32(1-q^2),
$$
differentiation gives
$$
q'=\frac12(1-q^2).
$$
Consequently,
$$
q(x)=\tanh\!\left(\frac{x+\sigma}{2}\right),\quad x\geq 0.
$$
for some $\sigma>0$. Therefore
$$
V(x)
 =1-\frac32\sech^2\!\left(\frac{x+\sigma}{2}\right),\quad x\geq 0.
$$

For the Dirichlet profile, the condition $V_D(0)=0$ gives
$$
\tanh\!\left(\frac{\sigma_D}{2}\right)
 =\frac1{\sqrt3}.
$$
Hence
$$
\sigma_D
 =2\tanh^{-1}\!\left(\frac1{\sqrt3}\right)
 =\log(2+\sqrt3).
$$

For the Robin profile, write
$$
q_R=\tanh\!\left(\frac{\sigma_R}{2}\right).
$$
The boundary condition $V_R'(0)=V_R(0)$ becomes
$$
3q_R(1-q_R^2)=3q_R^2-1,
$$
which is equivalent to
$$
3q_R^3+3q_R^2-3q_R-1=0.
$$

Let $G$ be the cubic polynomial defined by
$$
G(q)=3q(1-q^2)-(3q^2-1).
G\left(\frac{1}{\sqrt{3}}\right)>0,$$
and 
For $q\in \left(\frac{1}{\sqrt{3}},1\right)$,
$$
G'(q)=3-9q^2-6q<0.
$$
Moreover,

$$
G(1)<0.
$$
Thus, $G$ has a unique zero
$
q_R\in \left(\frac{1}{\sqrt{3}},1\right).
$
It follows that
$$
\sigma_R>\sigma_D.
$$
Since the profile is strictly increasing as a function 
$\sigma$, we obtain
$$
V_D(x)<V_R(x)<1,
\qquad x\ge0.
$$

Using $$V_N(x)=1,\quad x\geq 0,$$ we get
$$
\begin{aligned}
|V_D-V_R|+|V_N-V_R|
&=(V_R-V_D)+(1-V_R)\\
&=1-V_D.
\end{aligned}
$$
Furthermore,
$$
1-V_D(x)
 =\frac32\sech^2\!\left(\frac{x+\sigma_D}{2}\right),\quad x\geq 0.
$$
Since
$$
\sech^2 y
 =\frac{4e^{-2y}}{(1+e^{-2y})^2},\quad y\in (-\infty,\infty),
$$
we have
$$
1-V_B(x)
 =\frac{6e^{-(x+\sigma_B)}}
        {(1+e^{-(x+\sigma_B)})^2}
 \le6e^{-(x+\sigma_B)},\quad x\geq 0,
$$
for $B\in\{D,R\}$. Finally,
$$
1-V_B(x)
 =6e^{-(x+\sigma_B)}+O(e^{-2x})$$
as $x\to\infty.
$
This proves the result. $\blacksquare$

We can now see how the Neumann stationary profile and the Dirichlet stationary profile can be approached by a slightly more general Robin stationary profile.

\begin{proposition}
\label{prop:robin-coefficient}
For every positive number $\kappa$, the stationary problem
$$
V_\kappa''+V_\kappa(1-V_\kappa)=0,
$$
with
$$
V_\kappa'(0)=\kappa V_\kappa(0)$$
and
$$V_\kappa(x)\longrightarrow1$$
as $x\to\infty$
has a unique solution satisfying
$$
0<V_\kappa(x)<1,
\qquad x\ge0.$$

It is given explicitly by
$$
V_\kappa(x)
 =1-\frac32
 \sech^2\!\left(\frac{x+\sigma_\kappa}{2}\right),\quad x\geq0,$$
where
$$\sigma_\kappa
 =2\tanh^{-1}(q_\kappa)$$
and
$$q_\kappa\in \left(\frac{1}{\sqrt{3}},1\right)$$ is the unique zero of
$$
3q_\kappa(1-q_\kappa^2)
 =\kappa(3q_\kappa^2-1).
$$

For every fixed $x\ge0$, the map
$$
[0,\infty)\ni\kappa\longmapsto V_\kappa(x)\in [0,1]
$$
is strictly decreasing. Furthermore,
$$
V_\kappa\longrightarrow1$$ uniformly on $[0,\infty)$
as $\kappa\downarrow0$ and
$$
V_\kappa\longrightarrow V_D$$ locally uniformly on $[0,\infty)$
as $\kappa\to\infty.$
\end{proposition}
{\rm \textbf{Proof:}  By the calculations in the proof of
{\rm Corollary~\ref{cor:explicit-sharp-profiles}}, every stationary profile
with values in $(0,1)$ and limit $1$ has the form
$$
V(x)
 =1-\frac32\sech^2\!\left(\frac{x+\sigma}{2}\right),\quad x\geq 0,
$$
for some positive number $\sigma$. Let
$$
q=\tanh\!\left(\frac{\sigma}{2}\right).
$$
Then
$$V(0)=\frac{3q^2-1}{2}$$
and 
$$V'(0)=\frac{3}{2}2q(1-q^2).$$
Therefore the boundary condition
$$
V'(0)=\kappa V(0)
$$
is equivalent to
$$
3q(1-q^2)=\kappa(3q^2-1).
$$

Let $F(q,\kappa)$ be defined by
$$
F(q,\kappa)
 =3q(1-q^2)-\kappa(3q^2-1).
$$
For \(q\in\left(\frac{1}{\sqrt3},1\right)\)
$$
and
F_q(q,\kappa)
 =3-9q^2-6\kappa q<0.
$$
Moreover,
$$
F\left(\frac{1}{\sqrt3},\kappa\right)>0$$
and
$$F(1,\kappa)=-2\kappa<0.
$$
Thus, for every positive number $\kappa$, there exists a unique zero $q_\kappa$ with
$$
q_\kappa\in \left(\frac{1}{\sqrt3},1\right)
$$
and such that
$$
F(q_\kappa,\kappa)=0.
$$
The corresponding explicit profile satisfies the differential
equation, the Robin boundary condition, and the far-field condition.
The preceding representation also shows that every admissible profile
must arise from such a root. This proves the existence and uniqueness.

Implicit differentiation of
$$
F(q_\kappa,\kappa)=0
$$
gives
$$
\frac{dq_\kappa}{d\kappa}
 =-\frac{3q_\kappa^2-1}
        {9q_\kappa^2+6\kappa q_\kappa-3}<0.
$$
Consequently,
$$
\frac{d\sigma_\kappa}{d\kappa}
 =\frac{2}{1-q_\kappa^2}
   \frac{dq_\kappa}{d\kappa}<0.
$$
For every fixed $x\in [0,\infty)$, the function
$$
\sigma\longmapsto
1-\frac32\sech^2\!\left(\frac{x+\sigma}{2}\right)
$$
is strictly increasing. Hence
$$
\frac{\partial V_\kappa}{\partial\kappa}(x)<0.
$$

As $\kappa\downarrow0$, the defining equation implies that
$$
q_\kappa\longrightarrow1.
$$
Therefore
$$
\sigma_\kappa
 =2\tanh^{-1}(q_\kappa)
 \longrightarrow\infty
$$
and consequently
$$
V_\kappa\longrightarrow1
$$
locally uniformly on $[0,\infty)$. As $\kappa\to\infty$, division of the defining equation by
$\kappa$ gives
$$
\frac{3q_\kappa(1-q_\kappa^2)}{\kappa}
 =3q_\kappa^2-1.
$$
It follows that
$$
q_\kappa\longrightarrow\frac1{\sqrt3}$$
as
$\sigma_\kappa\longrightarrow\sigma_D.
$
Therefore
$$
V_\kappa\longrightarrow V_D
$$
locally uniformly on $[0,\infty)$. $\blacksquare$}

{\rm We can now see that the periodic response $\{P_B^\varepsilon(t):t\geq 0\}$ can be attracted exponentially.} 

\begin{theorem}
\label{thm:periodic-exponential-attraction}
Under the hypotheses of {\rm Theorem~\ref{thm:periodic-existence}}, let
$P_B^\varepsilon$ be the $T$-periodic lifted mild solution
constructed in {\rm Theorem~\ref{thm:periodic-existence}}. Let $\mu_B\in (0, \gamma_B).$
Then there exist positive constants
$
\varepsilon_B^{\mathrm{att}}$,
$\delta_B$ and
$
C_B^{\mathrm{att}}
$ such that for all positive numbers $\varepsilon$ with
$$
\varepsilon<\varepsilon_B^{\mathrm{att}},
$$
if the maximal lifted mild solution $U_B^\varepsilon$ with
the same boundary forcing as $P_B^\varepsilon$ satisfies 
%If
%$$
%d_0

$$
\|U_B^\varepsilon(0)-P_B^\varepsilon(0)\|_X\le\delta_B,
$$
then $U_B^\varepsilon(t)$ exists for every $t$ with $t\ge0$, and
$$
\|U_B^\varepsilon(t)-P_B^\varepsilon(t)\|_X
 \le
C_B^{\mathrm{att}}e^{-\mu_Bt}\|U_B^\varepsilon(0)-P_B^\varepsilon(0)\|_X,
\qquad t\ge0.
$$
\end{theorem}
%\end{proposition}
\textbf{Proof:}
Let $d$ be the function on $[0,\infty)$ defined by
$$
d(t)=U_B^\varepsilon(t)-P_B^\varepsilon(t),\quad t\geq  0.
$$
Since the two solutions satisfy the same boundary forcing, the function $d$
satisfies the homogeneous boundary condition in the domain of
$L_B$. Thus, by (\ref{LBScript}) the analysis similar to the  derivation of (\ref{NB(q)}),
$$
d'(t)=L_Bd+G_B^\varepsilon(t,d),\quad t\geq 0,
$$
where
$$
G_B^\varepsilon(t,d)
 =f(P_B^\varepsilon(t)+d)-f(P_B^\varepsilon(t))
   -f'(V_B)d,\quad t\geq 0.
$$

By the first-order estimate in Theorem~\ref{thm:periodic-existence}, there exists
a positive constant $C_{P,B}$ such that
$$
\sup_{t\in[0,T]}
\|P_B^\varepsilon(t)-V_B\|_X
\le C_{P,B}\varepsilon.
$$
Choose a closed neighborhood of the range of $V_B$ on which
$f''$ is bounded, and set
$$
H_B=\sup |f''|
$$
on this neighborhood. For sufficiently small positive numbers $r$ and $\varepsilon$, the arguments
below remain in this neighborhood whenever
$$
\|d\|_X\le r.
$$
Writing
$$
\begin{aligned}
G_B^\varepsilon(t,d)
&=
f(P_B^\varepsilon(t)+d)-f(P_B^\varepsilon(t))
-f'(P_B^\varepsilon(t))d\\
&\quad+
\bigl(
f'(P_B^\varepsilon(t))-f'(V_B)
\bigr)d
\end{aligned}
$$
for all $t\in [0,T],$
the mean-value theorem gives
$$
\|G_B^\varepsilon(t,d)\|_X
\le
H_B
\left(
C_{P,B}\varepsilon+\|d\|_X
\right)
\|d\|_X,\quad t\in [0,T].
$$
Let $r$ be a sufficiently small positive number. By decreasing
$\varepsilon_B^{\mathrm{att}}$, if necessary, so that
$$
\varepsilon_B^{\mathrm{att}}\le\varepsilon_B
$$
and defining $\vartheta_B$ by
$$
\vartheta_B
 =
\frac{
M_BH_B
\left(
C_{P,B}|\varepsilon|+r
\right)
}{
\gamma_B-\mu_B
}
\le\frac12
$$
whenever
$$
|\varepsilon|<\varepsilon_B^{\mathrm{att}}.
$$

On every time interval with
$$
\|d(t)\|_X\le r,
$$
the formula for variation-of-constants  gives
$$
d(t)
 =S_B(t)d_0
 +\int_0^t
 S_B(t-s)G_B^\varepsilon(s,d(s))\,ds,\quad t\geq 0,
$$
where $d_0=U_B^\varepsilon(0)-P_B^\varepsilon(0).$
Let $Z$ be the function on  $[0,\infty)$ defined by
$$
Z(t) 
 =\sup_{0\le s\le t}
 e^{\mu_Bs}\|d(s)\|_X,\quad t\geq 0.
$$
Using
$$
\|S_B(t)\|_{\mathcal L(X)}
\le M_Be^{-\gamma_Bt}, \quad t\geq  0,
$$
we obtain
$$
\begin{aligned}
e^{\mu_Bt}\|d(t)\|_X
&\le
M_B\|d_0\|_X\\
&\quad+
M_BH_B
\left(
C_{P,B}\varepsilon+r
\right)
\int_0^t
e^{-(\gamma_B-\mu_B)(t-s)}
e^{\mu_Bs}\|d(s)\|_X\,ds
\end{aligned}
$$
for all $t\in [0,\infty).$
Consequently,
$$
Z(t)
\le
M_B\|d_0\|_X+\vartheta_BZ(t)
\le
M_B\|d_0\|_X+\frac12Z(t),\quad t\geq 0.
$$
Therefore
$$
Z(t)\le2M_B\|d_0\|_X,\quad t\geq 0.
$$

Choose $\delta_B$ with
$$
0<\delta_B\le\frac{r}{4M_B}.
$$
If
$$
\|d_0\|_X\le\delta_B,
$$
then
$$
\|d(t)\|_X
\le
2M_Be^{-\mu_Bt}\|d_0\|_X
\le\frac r2,t\geq 0.
$$
This estimate completes the continuation argument. Hence the
solution remains in the chosen neighborhood and extends to every
$t$ in $[0,\infty)$. Moreover,
$$
\|U_B^\varepsilon(t)-P_B^\varepsilon(t)\|_X
\le
2M_Be^{-\mu_Bt}
\|U_B^\varepsilon(0)-P_B^\varepsilon(0)\|_X\quad t\geq 0.
$$
Thus the conclusion holds with
$$
C_B^{\mathrm{att}}=2M_B.
$$ and the proof is complete. $\blacksquare$

We end this section with a quantitative Neumann--Robin difference.

\begin{corollary}
\label{cor:quantitative-autonomous-NR}
Let $f$ be the function on  $[0,1]$ given by
$$
f(u)=u(1-u),\quad 0\leq u\leq 1.$$
Let $U_N$ and $U_R$ be solutions of the autonomous Neumann and
Robin problems, respectively, and assume
$$
U_N(0)-1\in X,$$
and
$$U_R(0)-V_R\in X.
$$
Furthermore, suppose that there exist positive constants
$
\delta$,
$C$
and
$\mu
$
such that if
$$
\|U_N(0)-1\|_X
+\|U_R(0)-V_R\|_X
\le\delta,
$$
then
$$
\begin{aligned}
&\|U_N(t)-U_R(t)-(1-V_R)\|_X\\
&\qquad\le
Ce^{-\mu t}
\left(
\|U_N(0)-1\|_X
+\|U_R(0)-V_R\|_X
\right),
\qquad t\ge0.
\end{aligned}
$$
In particular, this gives a genuine temporal decay rate for the
Neumann--Robin comparison error in the near-equilibrium regime.
\end{corollary}
\textbf{Proof:}
For the logistic Neumann and Robin problems, the hypotheses of
Theorem~\ref{thm:periodic-exponential-attraction} are verified by
Proposition~\ref{prop:linearized-stability} and Corollary~\ref{cor:logistic-periodic-forcing}

Applying Theorem~\ref{thm:periodic-exponential-attraction} with
$\varepsilon=0$ separately to the Neumann and Robin problems, we get
$$
P_N^0=V_N\equiv1,
\qquad
P_R^0=V_R.
$$
Therefore for sufficiently small initial perturbations, there exist
positive constants $C_N$, $C_R$, $\mu_N$ and $\mu_R$ such that
$$
\|U_N(t)-1\|_X
\le
C_Ne^{-\mu_Nt}\|U_N(0)-1\|_X,\quad t\geq 0,
$$
and
$$
\|U_R(t)-V_R\|_X
\le
C_Re^{-\mu_Rt}\|U_R(0)-V_R\|_X,\quad t\geq 0.
$$

Since
$$
\begin{aligned}
U_N(t)-U_R(t)-(1-V_R)
&=(U_N(t)-1)-(U_R(t)-V_R),
\end{aligned}
$$
for  all $t\in [0,\infty),$
the triangle inequality gives
$$
\begin{aligned}
&\|U_N(t)-U_R(t)-(1-V_R)\|_X\\
&\quad\le
C_Ne^{-\mu_Nt}\|U_N(0)-1\|_X
+C_Re^{-\mu_Rt}\|U_R(0)-V_R\|_X.
\end{aligned}
$$
for all $t\in [0,\infty).$
Let $\mu$ and $C$ be given by
$$
\mu=\min\{\mu_N,\mu_R\}$$
and
$$
C=\max\{C_N,C_R\}.
$$
By choosing a positive number $\delta$ smaller than the two admissible
perturbation radii yields the asserted estimate. $\blacksquare$
\section*{Declaration of AI-assisted tools}

During the preparation of this manuscript, the first author has used ChatGPT as an auxiliary tool for LaTeX formatting and for assistance in drafting and organizing some proof arguments based on ideas provided by the first author. All AI-assisted material, together with the mathematical arguments, references, and the final manuscript, have been independently reviewed, revised, and checked by both authors. Both authors take full responsibility for the accuracy and integrity of the work.

\section*{Acknowledgements}

The first author is grateful to Professor Woldegebriel Assefa Woldegerima, whose undergraduate 4000-level reading course in mathematics has introduced the Fisher-KPP equation initiating the research in  this paper. Both authors are grateful to Professor Dmitry Pelinovsky for his careful reading of an earlier version of the manuscript and for the detailed and constructive comments, which have helped clarify the analysis, make precise several assumptions, and improve the overall presentation of the paper.

\end{document}